\documentclass[11pt]{amsart}
\usepackage{graphicx}
\usepackage[space]{grffile}
\usepackage[T1]{fontenc}
\usepackage{color}
\usepackage[colorlinks,linkcolor=blue,citecolor=blue, pdfstartview=FitH]{hyperref}
\usepackage{backref}
\usepackage{amssymb}
\usepackage{epsf}
\usepackage{caption}
\usepackage{slashed}

\def\p{\partial}
\def\o{\overline}

\def\mb{\mathbb}
\def\mc{\mathcal}
\def\mr{\mathrm}

\def\n{\nabla}
\def\wt{\widetilde}

\def\s{\slashed}

\theoremstyle{plain}

\theoremstyle{plain}
\newtheorem{theorem}{Theorem}[section]
\theoremstyle{plain}
\newtheorem{definition}[theorem]{Definition}
\theoremstyle{plain}
\newtheorem{lemma}[theorem]{Lemma}
\theoremstyle{remark}
\newtheorem{remark}[theorem]{Remark}
\theoremstyle{plain}
\newtheorem{proposition}[theorem]{Proposition}
\numberwithin{equation}{section}
\theoremstyle{plain}
\newtheorem{example}[theorem]{Example}

\makeatletter
\@namedef{subjclassname@2020}{\textup{2020} Mathematics Subject Classification}
\makeatother

\begin{document}

\title{Band width estimates of CMC initial data sets}

\author{Xiaoxiang Chai}
\address{Xiaoxiang Chai: Korea Institute for Advanced Study, Seoul 02455, South Korea}
\email{xxchai@kias.re.kr}

\author{Xueyuan Wan}
\address{Xueyuan Wan: Mathematical Science Research Center, Chongqing University of Technology, Chongqing 400054, China.}
\email{xwan@cqut.edu.cn}


\begin{abstract}
	 We generalize a band width estimate of Gromov to CMC initial data sets. We
  give three independent proofs: via the stability of a hypersurface with
  prescribed null expansion, via a perturbation of the spacetime harmonic
  function and via the Dirac operator.
\end{abstract}

 
  \thanks{ Research of Xiaoxiang Chai is
  supported by KIAS Grants under the research code MG074402. Research of Xueyuan Wan is partially supported by the National Natural Science Foundation of China (grant No. 12101093) and the Scientific Research Foundation of the Chongqing University of Technology.}
\maketitle


\section{Introduction}

An initial data set $(M, g, k)$ is an $n$-dimensional Riemannian manifold $(M,
g)$ immersed as a spacelike hypersurface in an $(n + 1)$-dimensional
Lorent-zian spacetime $\mathcal{S}^{n, 1}$ where $k (X, Y) = \langle
\nabla^{\mathcal{S}}_X e_0, Y \rangle$ is the second fundamental form with
respect to a choosed unit timelike normal $e_0$. We call $M$ a
constant mean curvature (or in short CMC) initial data set if $M$ is of constant mean curvature
in $\mathcal{S}^{n, 1}$, that is, $\ensuremath{\operatorname{tr}}_g k$ is a
constant. The Einstein tensor $G$ of $\mathcal{S}^{n, 1}$ when restricted to
$M$ gives the constraint equation
\[ 2 \mu : = G_{0 0} = R_g - |k|_g^2 + (\ensuremath{\operatorname{tr}}_g k)^2,
\]
and
\[ J_i : = G_{0 i} = (\ensuremath{\operatorname{div}}_g k - \mathrm{d}
   (\ensuremath{\operatorname{tr}}_g k))_i, \]
where $R_g$ is the scalar curvature of $(M, g)$, $\mu$ and $J$ are
respectively called energy density and current density.

Given a two-sided hypersurface $\Sigma$ in $M$, we denote by $p$ the second
fundamental form computed with respect to the choosed normal $\nu$. Further,
we denote $H =\ensuremath{\operatorname{div}}_{\Sigma} \nu$ the mean
curvature. The quantity
\[ \theta^{\pm} = \pm H +\ensuremath{\operatorname{tr}}_{\Sigma} k
   \label{expansion} \]
is called outward (inward) \text{{\itshape{null expansion}}} of the
hypersurface $\Sigma$. A hypersurface with vanishing null expansion $\theta^+$
($\theta^-$) is called a marginally outer (inner) trapped hypersurface or in
short MOTS (MITS). We will be concerned with only the outward null expansion
and we use the shorthand $\theta = \theta^+$.

Let $M =\mathbb{T}^{n - 1} \times [- 1, 1]$, $\partial_{\pm} M=\mathbb{T}^{n -
1} \times \{\pm 1\}$, $\theta_+$ be the expansion at $\partial_+M$ computed
with respect to the outward normal $\nu_+$ and $\theta_-$ be the null
expansion computed with respect to the inward normal $\nu_-$. It is easy to
see that this choice of normal vector field is consistent in the sense that
$\langle \nu_{\pm}, \frac{\partial}{\partial t} \rangle > 0$ at both $\partial_{\pm} M$. Here
$t$ is the coordinate representing the segment $[- 1, + 1]$. We denote the
expansion at $\partial_{\pm}M$ computed in this way by $\theta_{\pm}$. The
\text{{\itshape{width}}} of $M$ is define to be the distance between two
boundaries $\partial_{\pm}M$, which is denoted by $\mathrm{width}(M,g)$.

Our main result is a width estimate for a CMC initial data set $(M, g, k)$.
Before stating the main result, we define some quantities we frequently use.
Let $\sigma$, $\lambda$ be two real constants, define $\eta = \eta (t)=\eta_0(t)$ to be
the solution to the ordinary differential equation
\begin{equation}
  \sigma + \tfrac{n}{n - 1} \eta^2 - 2 \eta \lambda + 2 \eta' = 0,\quad \eta'<0. \label{ode 0}
\end{equation}
The solution of \eqref{ode 0} is given in the Appendix. The explicit form of
$\eta$ depends on the sign of $\sigma -\tfrac{n - 1}{n} \lambda^2$, see respectively \eqref{eta equal}, \eqref{eta greater} and \eqref{eta less}. Define
the interval $[r_-, r_+]$ by setting
\[ r_{\pm} = \pm \tfrac{(n - 1) \pi}{n} \left[ \tfrac{n - 1}{n} (\sigma -
   \tfrac{n - 1}{n} \lambda^2) \right]^{- \tfrac{1}{2}} \]
if $\sigma > \tfrac{n - 1}{n} \lambda^2$ and otherwise fixing $r_{\pm}$ with
the property $0 \not\in [r_-, r_+]$.

Now we are ready to state the theorem.

\begin{theorem}
  \label{width estimate positive}Let $M =\mathbb{T}^{n - 1} \times [- 1, 1]$,
  assume on $(M, g, k)$ that $\ensuremath{\operatorname{tr}}_g k = \lambda$,
  \begin{equation}
    \mu - |J| \geq \tfrac{1}{2} \sigma \label{dominant scalar bound}
  \end{equation}
  and the null expansions at the boundaries $\partial_{\pm}M$ satisfy
  \begin{equation}
    \theta_- \leq \eta (t_-), \theta_+ \geq \eta (t_+),
    \label{expansion bound}
  \end{equation}
  for some $t_{\pm}$ with $r_- < t_- < t_+ < r_+$. Then
  \begin{equation}
    \ensuremath{\operatorname{width}} (M, g) \leq t_+ - t_- .
    \label{explicit width estimate positive}
  \end{equation}
\end{theorem}
\begin{remark}
In fact, using the method of the Dirac operator,	 we can prove Theorem \ref{width estimate positive} when $M$ is a spin band of infinite vertical $\widehat{\mr{A}}$-area or a $\mc{K}\mc{O}$-band, see Theorems \ref{thminfA}, \ref{thmKO}.
\end{remark}

Now we briefly discuss the history and previous results related to the Theorem
\ref{width estimate positive}.
Gromov {\cite{gromov-metric-2018}} established via torical symmetrization the
width estimate for the case $k = 0$, $\sigma = n (n - 1)$ with the metric in
Example \ref{band of positive scalar} as a rigid band in the sense that
\eqref{dominant scalar bound}, \eqref{expansion bound} and \eqref{explicit
width estimate positive} are equalities.

\begin{example}[{\cite{gromov-metric-2018}}]
  \label{band of positive scalar}The metric $g = \mathrm{d} t^2 + \cos
  (\tfrac{n t}{2})^{\tfrac{4}{n}} \tau$ where $\tau$ is the standard metric on
  the torus $\mathbb{T}^{n - 1}$ has scalar curvature $R_g = n (n - 1)$. The
  mean curvature of each $t$-level set is $H = - (n - 1) \tan (\tfrac{n
  t}{2})$ and satisfies
  \begin{equation}
    n (n - 1) + \tfrac{n}{n - 1} H^2 + 2 H' = 0.
  \end{equation}
\end{example}

Zhu {\cite{zhu-width-2021}}
established an optimal band width estimate for manifold with constant positive
Ricci curvature lower bound via $\mu$-bubble. Zhu's method was further applied by {\cite{rade-scalar-2021}} to prove the band width estimate under scalar curvature lower bound. The spinorial proof was developed by 
{\cite{zeidler-band-2020, cecchini-scalar-2021-arxiv}}. 

There are other two examples of rigid bands. The case $k \equiv 0$,
$\sigma < 0$ is also due to Gromov.

\begin{example}[\cite{gromov-scalar-2019-arxiv}]
  \label{band of negative scalar}The metric $\mathrm{d} t^2 + \sinh (\tfrac{n
  t}{2})^{\tfrac{4}{n}} \tau$ on $(0, \infty) \times \mathbb{T}^{n - 1}$ has
  scalar curvature $- n (n - 1)$ and each $t$-level set has mean curvature $H
  = (n - 1) \coth (\tfrac{n t}{2})$ satisfying
  \begin{equation}
    - n (n - 1) + \tfrac{n}{n - 1} H^2 + 2 H' = 0.
  \end{equation}
  The metric is a rigid band with $k = 0$, $\sigma = - 2 n (n - 1)$.
\end{example}

\begin{example}
  \label{scalar flat band}The metric $g = \mathrm{d} t^2 + (\tfrac{n
  t}{2})^{\tfrac{4}{n}} \tau$ on $(0, \infty) \times \mathbb{T}^{n - 1}$ is of
  zero scalar curvature, and each $t$-level set has mean curvature $H =
  \tfrac{2 (n - 1)}{n t}$ and satisfies
  \begin{equation}
    R_g + \tfrac{n}{n - 1} H^2 + 2 H' = 0
  \end{equation}
  is a rigid band with $k \equiv 0$ and $\sigma = 0$.
\end{example}

The article is organized as follows:

In Section \ref{pne proof}, we generalize the $\mu$-bubble to the spacetime
settings, which we call a hypersurface of prescribed null expansion (see
Definition \ref{pne def}) and use it to give a proof of Theorem \ref{width
estimate positive} in dimensions less than eight. We also discuss some easy
generalizations of Theorem \ref{width estimate positive} where the initial
data set is not CMC. In Section \ref{spacetime harmonic function proof}, we
restrict to dimension three only and use a perturbation of spacetime harmonic
functions to give a simple proof. We can show rigidity in this case. In
Section \ref{spinor proof}, we give the proof via the Dirac operator.

\

\noindent\text{{\bfseries{Acknowledgments}}} The first author would like to thank Tin-Yau Tsang (UCI) for discussing about the paper \cite{lee-density-2022-arxiv}.

\

\section{Proof via hypersurfaces of prescribed null expansion}\label{pne
proof}

Gromov {\cite{gromov-metric-2018,gromov-four-2021}} studied stable
$\mu$-bubbles in the study of scalar curvature, which leads to many interesting
results in positive scalar curvature other than band width estimates. For
example, positive mass theorem for asymptotically flat manifolds with
arbitrary ends {\cite{lesourd-positive-2021}}, nonexistence of metrics with
positive scalar curvature {\cite{chodosh-generalized-2020-arxiv}} on
aspherical manifolds, see also {\cite{gromov-four-2021}} for a survey of more
results.

A $\mu$-bubble is just a hypersurface of prescribed mean curvature, and the
use of $\mu$-bubble in scalar curvature problems dates back to minimal surface
techniques used by Schoen-Yau settled the positive mass conjecture
{\cite{schoen-proof-1979}} and {\cite{schoen-existence-1979}}. The analog of
minimal surfaces in initial data sets is a marg\text{{\upshape{}}}inally outer
trapped surface (see
{\cite{andersson-local-2005}}).

A MOTS arises boundaries of blow up sets of Jang equation
{\cite{schoen-proof-1981}}. The existence of MOTS was proven by
{\cite{andersson-area-2009,eichmair-plateau-2009}}. It was applied to
establish the spacetime positive mass theorem
{\cite{eichmair-spacetime-2016}}. Motivated by the $\mu$-bubble, we propose
the notion of a hypersurface of prescribed null expansion.

\begin{definition}
  \label{pne def}Given a smooth function $p$ on $(M, g, k)$, a hypersurface
  $\Sigma$ is called a hypersurface of prescribed null expansion if
  \begin{equation}
    \theta = p \label{pne}
  \end{equation}
  along $\Sigma$. We say that $\Sigma$ is stable if there exists a vector
  field $X = \varphi \nu$ with nonzero $\varphi \geq 0$ such that the variation of $\theta-p$ is nonnegative along $X$, that is,
  \begin{equation}
    \delta_X (\theta - p) \geq 0. \label{stability}
  \end{equation}
\end{definition}
The definition is motivated by both the stability of MOTS {\cite{andersson-local-2005}}
and the stability of the $\mu$-bubble.
The case $p \equiv 0$ gives the definition of MOTS and the case $k \equiv 0$
recovers surfaces of prescribed mean curvature or in Gromov's terminology a
$\mu$-bubble.  
We can write down the stability \eqref{stability} explicitly by calculating the variation of $\theta-p$.
\begin{lemma}
  \label{stability operator} If $\Sigma$ is a stable hypersurface of prescribed null expansion $p$ in the
  initial data set $(M, g, k)$, then there exists a nonzero function $\varphi
  \geq 0$ such that $L \varphi \geq 0$ where $L$ is given by
\begin{align}
  \begin{split}
 L \varphi &= - \Delta_\Sigma \varphi + 2 \langle W, \nabla^\Sigma \varphi \rangle +
(\ensuremath{\operatorname{div}}W - W + Q) \varphi \\
& \quad - \tfrac{1}{2} (p^2 - 2 p\ensuremath{\operatorname{tr}}_g k + 2
  \nu (p)) \varphi,
  \end{split}\label{linearization}
\end{align}
where $W$ is the vector field tangential to $\Sigma$ which is dual to $k (\nu,
\cdot)$,
\[
  Q = \tfrac{1}{2} R_{\Sigma} - \mu - J (\nu) - \tfrac{1}{2} | \chi |^2.
\] Here $\chi=h+k$ is called the null second fundamental form or the shear tensor. We call $L$ the stability operator.
\end{lemma}
\begin{proof}
From Definition \ref{pne def}, there exists a nonzero vector field $\varphi
\nu$ such that $\delta_{\varphi \nu} (\theta - p) \geq 0$ with $\varphi
\geq 0$. It suffices to prove
\begin{equation}
  \delta_{\varphi \nu} (\theta - p) = L \varphi .
\end{equation}
Now we calculate $\delta_{\varphi \nu} (\theta - p)$. The variation of
$\delta_X \theta$ was already done in {\cite[(10)]{andersson-jangs-2010}}.
They derived that
\begin{equation}
  \delta_X \theta = - \Delta_\Sigma \varphi + 2 \langle W, \nabla^\Sigma \varphi \rangle +
  (\ensuremath{\operatorname{div}}W - W + Q) \varphi - \tfrac{1}{2} \theta
  (\theta - 2\ensuremath{\operatorname{tr}}_g k) .
\end{equation}
With $X (p) = \varphi \langle \nabla p, \nu \rangle$ and $\theta = p$, we obtain
$\delta_X (\theta - p) = L \varphi$ where $L \varphi$ is given by
\eqref{linearization}.
\end{proof}
From Definition \ref{pne def},  a hypersurface $\Sigma$ of prescribed null expansion is stable
  if and only if the principal eigenvalue $\lambda_1$ of $L$ is nonnegative.
The operator $L$ is not necessarily self-adjoint; hence its eigenvalues can
have complex values. But it does have a real eigenvalue $\lambda_1$ and with
least real part among all eigenvalues of $L$ by Krein-Rutman theorem (see
{\cite{andersson-local-2005}}). The eigenvalue $\lambda_1$ is called the
\text{{\itshape{principal eigenvalue}}} and its eigenfunction can be chosen
strictly positive.

Let $\chi^0$ be the trace free part of the null second fundamental form, then
\begin{equation}
  | \chi |^2 = | \chi^0 |^2 + \tfrac{1}{n - 1}
  (\ensuremath{\operatorname{tr}}_g \chi)^2 = | \chi^0 |^2 + \tfrac{1}{n - 1}
  p^2 
\end{equation}
by \eqref{pne}. So the stability operator can be written in the following form:
\begin{align}
\begin{split}
L \varphi & = - \Delta_{\Sigma} \varphi + 2 \langle W, \nabla^{\Sigma}
\varphi \rangle + ({\operatorname{div}}W - W) \varphi \\
&\quad + (\tfrac{1}{2} R_{\Sigma} - \tfrac{1}{2} | \chi^0 |^2) \varphi -
[\mu + J (\nu) + \tfrac{1}{2} (\tfrac{n}{n - 1} p^2 - 2
p{\operatorname{tr}}_g k + 2 \nu (p))] \varphi . 
\end{split} \label{new linearization}
\end{align}
Because of the above and {\cite{galloway-generalization-2006}}, 
we introduce the relaxed energy dominant energy condition.
\begin{definition}
  We say that an initial data set $(M, g, k)$ satisfies the relaxed dominant
  energy condition if there is a smooth function $p$ defined on $M$ such that
  \begin{equation}
    \mu - |J| + \tfrac{1}{2} (\tfrac{n}{n - 1} p^2 - 2
    p\ensuremath{\operatorname{tr}}_g k - 2| \nabla p|) \geq 0.
    \label{modified dec}
  \end{equation}
\end{definition}

We have the following classifications of the topology of a stable
hypersurfaces of prescribed null expansion.

\begin{theorem}
  \label{topology}Assume that $\Sigma$ is a stable hypersurface of prescribed
  null expansion $p$ in an initial data set $(M, g, k)$ satisfying
  \eqref{modified dec}, then $\Sigma$ is of positive Yamabe type unless
  $\Sigma$ is Ricci flat, $\chi^0$ vanishes and
  \begin{equation}
    \mu + J (\nu) + \tfrac{1}{2} (\tfrac{n}{n - 1} p^2 - 2
    p\ensuremath{\operatorname{tr}}_g k + 2 \nu (p)) \label{modified mu J}
  \end{equation}
  vanishes along $\Sigma$.
\end{theorem}

\begin{proof}
  The proof proceeds as {\cite{galloway-generalization-2006}} using the
  relaxed dominant energy condition \eqref{modified dec} and the stability
  \eqref{stability}. We refer the readers to
  {\cite{galloway-generalization-2006}} for the details.
\end{proof}

\begin{remark}\label{transfer}
  In fact, let $\tilde{k} = k - \tfrac{1}{n - 1} p g$, then a hypersurface $\Sigma$ of
  prescribed null expansion is a MOTS in the new initial data set $(M, g,
  \tilde{k})$. The stability operator \eqref{linearization}, \eqref{new linearization}, the energy condition \eqref{modified dec}, Theorem
  \ref{topology} and the existence of hypersurface with prescribed null
  expansion can be readily obtained by transferring to $(M, g, \tilde{k})$. We
  will use this fact about the existence in proving Theorem \ref{width estimate positive} via
  constructing a hypersurface of prescribed null expansion. See \cite[Section 6]{lee-density-2022-arxiv}.
\end{remark}

\begin{lemma}
  \label{construction of h}Let $(M, g, k)$ be as in Theorem \ref{width
  estimate positive}, if the width of $M$ is greater than $t_+ - t_-$, then
  there exists a smooth function $p$ such that
  \begin{equation}
    \mu - |J| + \tfrac{1}{2} (\tfrac{n}{n - 1} p^2 - 2
    p\ensuremath{\operatorname{tr}}_g k - 2| \nabla p|) > 0 \label{most strict dec}
  \end{equation}
  in $M$ and
  \begin{equation}
    \theta_- - p (\p_-M) < 0 \text{ and } \theta_+ - p (\p_+M) > 0. \label{strict expansion bound}
  \end{equation}
\end{lemma}

\begin{proof}
  By the assumption on the width of $M$, we fix a small $\varepsilon > 0$ such
  that $\ensuremath{\operatorname{width}} (M, g) > t_+ - t_- + 2 \varepsilon$.
  We use a lemma due to Zhu {\cite[Lemma 4.1]{zhu-width-2021}} to get a
  surjective smooth function
  \[ \phi : (M, g) \to [t_- - \varepsilon, t_+ + \varepsilon] \]
  with $\ensuremath{\operatorname{Lip}} (\phi) < 1$ and such that $\phi^{- 1}
  (t_- - \varepsilon) = \partial_-M$ and $\phi^{- 1} (t_+ + \varepsilon) =
  \partial_+M$. We define
\begin{equation*}
  p(x)=\eta\circ\phi(x),\quad x\in M
\end{equation*}
  where $\eta=\eta(t)$ is the function satisfying the ordinary differential equation
  \eqref{ode} and $\eta'(t)<0$. Then
\begin{align}
\begin{split}
& \tfrac{n}{n - 1} p^2 - 2 p\ensuremath{\operatorname{tr}}_g k - 2 |\nabla p|
\\
\geq & \tfrac{n}{n - 1} p^2 - 2 p\ensuremath{\operatorname{tr}}_g k -
2 | \eta' | \ensuremath{\operatorname{Lip}} (\phi) \\
= & - \sigma - 2 \eta' \circ \phi + 2 | \eta' \circ \phi |
\ensuremath{\operatorname{Lip}} (\phi) \\
= & - \sigma - 2 \eta' \circ \phi (1 -\ensuremath{\operatorname{Lip}}
\phi) \\
> & - \sigma,
\end{split}
\end{align}
  where the last inequality follows from $\eta' < 0$. So \eqref{most strict
  dec} is true because of \eqref{dominant scalar bound}. Furthermore, since
  the function $\eta$ is monotonically decreasing, we have the strict
  expansion bound \eqref{strict expansion bound} of the boundaries from the
  bounds \eqref{expansion bound}.
\end{proof}

Now we can prove Theorem \ref{width estimate positive}.

\begin{proof}[Proof of Theorem \ref{width estimate positive} when $\dim M
\leq 7$]
  We argue by contradiction, assume that there exists a small $\varepsilon >
  0$ such that $\ensuremath{\operatorname{width}} (M, g) > t_+ - t_- + 2
  \varepsilon$. We can invoke Lemma \ref{construction of h} to construct an
  $p$ satisfying \eqref{most strict dec} and \eqref{strict expansion bound}.
  
  By the bounds \eqref{strict expansion bound}, we can apply Eichmair
  {\cite{eichmair-plateau-2009}} (see Remark \ref{transfer}) to find a closed hypersurface $\Sigma$ of
  prescribed null expansion $p$ which is homologous to $\partial_{\pm}M$ and
  stable. Note that also $\partial_{\pm}M \cap \Sigma = \emptyset$. Because of
  \eqref{most strict dec} in Lemma \ref{construction of h}, we can apply
  Theorem \ref{topology} to conclude that $\Sigma$ must admit a metric of
  positive scalar curvature. However, this contradicts the construction of
  $\Sigma$ which says that $\Sigma$ is topologically $\mathbb{T}^{n - 1}$.
\end{proof}

\begin{remark}
  In fact, Theorems \ref{width estimate positive} can be generalized to the
  cases with a general warped product. See {\cite{rade-scalar-2021}}. The proofs are similar, and we leave the
  details to the readers. A general version is also available via the spinor methods \cite{cecchini-scalar-2021-arxiv}.
\end{remark}

We are not able to show a rigidity. We wish to apply {\cite[Theorem
1.2]{eichmair-initial-2021}} for the initial data set $(M, g, k - \tfrac{1}{n
- 1} \eta\circ\phi g)$ with $\phi$ constructed later in Lemma \ref{rigid width function}. However, $\eta\circ\phi$ is only Lipschitz. The
paper {\cite{eichmair-initial-2021}} used {\cite[Theorem
5.1]{andersson-area-2009}} in which they required that $\eta\circ\phi$ can be
differentiated at least twice.

The rigidity is true for $k = \tfrac{1}{n - 1} \lambda g$ via spinorial
techniques {\cite{cecchini-scalar-2021-arxiv}} because the dominant energy
scalar $\mu - |J| \geq \tfrac{1}{2} \sigma$ is just a scalar curvature
bound, the bounds \eqref{expansion bound} are just bounds on the mean
curvatures of $\partial_{\pm}M$ and everything reduces to the cases by
{\cite{gromov-metric-2018}}, {\cite{cecchini-scalar-2021-arxiv}} and
{\cite{rade-scalar-2021}}.

We mention how to generalize of the Theorems \ref{width estimate positive}
when $\ensuremath{\operatorname{tr}}_g k$ is not constant. Let $\lambda =
\inf_M \ensuremath{\operatorname{tr}}_g k$ and $\Lambda = \sup_M
\ensuremath{\operatorname{tr}}_g k$. For example, we require that $\eta
(t_{\pm}) > 0$, then
\begin{equation}
  - 2 \eta \ensuremath{\operatorname{tr}}_g k \geq - 2 \eta \Lambda .
\end{equation}
We get a solution $\eta$ of \eqref{ode 0} by replacing $\lambda$ by $\Lambda$.
Then we can state a slight generalization of Theorem \ref{width estimate
positive}. In the following theorem, we require that $\sigma > \tfrac{n -
1}{n} \Lambda^2$.

\begin{theorem}
  Given real constants $\Lambda$, $\sigma$ with $\sigma > \tfrac{n - 1}{n}
  \Lambda^2$, define
  \[ \eta (t) = \tfrac{n - 1}{n} \Lambda - \sqrt{\tfrac{n - 1}{n} \left(
     \sigma - \tfrac{n - 1}{n} \Lambda^2 \right)} \tan \left( \tfrac{n}{2 (n -
     1)} \sqrt{\tfrac{n - 1}{n} (\sigma - \tfrac{n - 1}{n} \Lambda^2)} t
     \right) . \]
  Let $M =\mathbb{T}^{n - 1} \times [- 1, 1]$, assume on $(M, g, k)$ that
  $\ensuremath{\operatorname{tr}}_g k \leq \Lambda$,
  \[ \mu - |J| \geq \tfrac{1}{2} \sigma \]
  and the null expansions at the boundaries $\partial_{\pm}M$ satisfy
  \[ \theta_- \leq \eta (t_-), \theta_+ \geq \eta (t_+), \]
  where $t_{\pm}$ are two numbers with $\eta (t_{\pm}) > 0$. Then
  \[ \ensuremath{\operatorname{width}} (M, g) \leq t_+ - t_- . \]
\end{theorem}

\section{Proof via a perturbation of the spacetime harmonic
function}\label{spacetime harmonic function proof}

Stern {\cite{stern-scalar-2019}} initiated an approach using a harmonic map to
prove that there does not exist a metric of positive scalar curvature on
$\mathbb{T}^3$. The approach has been generalized to the spacetime by
{\cite{bray-harmonic-2022,hirsch-spacetime-2020}} to show a spacetime positive mass theorem. See
{\cite{bray-spacetime-2021}} more related results. We assume that $M
=\mathbb{T}^2 \times [- 1, 1]$ in this section. We recall the integral
inequality.

\begin{lemma}[{\cite{hirsch-spacetime-2020}}]
  If
  \begin{equation}
    \Delta_g u +\ensuremath{\operatorname{tr}}_g k |\nabla u| = 0 \label{spacetime
    harmonic function}
  \end{equation}
  on $M$, then $u \in C^{2,\alpha} \cap W^{3, p}_{\ensuremath{\operatorname{loc}}}$ and
  satisfies
  \begin{equation}
    \int_{\partial_{\pm M}} \pm \partial_{\nu \pm} |\nabla u| + k (\nabla u, \pm
    \nu_{\pm}) \geq \int_{\underline{u}}^{\bar{u}} \mathrm{d} s \int_{\Sigma_s}
    [\tfrac{1}{2} \tfrac{|\nabla^2 u - k_{i j} |\nabla u||^2}{|\nabla u|^2} + \mu + J
    (\tfrac{\nabla u}{|\nabla u|}) - K_{\Sigma_s}], \label{integral inequality of
    spacetime harmonic function}
  \end{equation}
  where $\Sigma_s$ are level sets of $u$, $K_{\Sigma_s}$ is its Gauss
  curvature, $\bar{u} = \sup_M u$ and $\underline{u} = \inf_M u$.
\end{lemma}

\begin{proof}[Proof of Theorem \ref{width estimate positive} in dimension 3]
  Let $p$ be the function constructed in Lemma \ref{construction of h},
  $\tilde{k} := k - \tfrac{1}{2} p g$ and $u$ be a solution of
  \begin{equation}
    \Delta_g u +\ensuremath{\operatorname{tr}}_g \tilde{k} |\nabla u| = \Delta_g u
    + (\ensuremath{\operatorname{tr}}_g k - \tfrac{3}{2} p) |\nabla u| = 0 \text{
    in } M \label{perturbation of spacetime harmonic}
  \end{equation}
  and $u = {\pm}1$ on $\partial_{\pm}M$. Replacing the $k$, $\mu$, $J$ by
  $\tilde{k}$, $\tilde{\mu}$, $\tilde{J}$ in \eqref{integral inequality of
  spacetime harmonic function}, by an easy calculation,
  \begin{equation}
    \tilde{\mu} = \mu + \tfrac{1}{2} (\tfrac{3}{2} p^2 - 2
    p\ensuremath{\operatorname{tr}}_g k),\quad \tilde{J} = J + \nabla p.
    \label{perturbed mu J}
  \end{equation}
  So
\begin{align}
\begin{split}
& \tilde{\mu} + J (\tfrac{\nabla u}{|\nabla u|}) \\
= & \mu + J (\tfrac{\nabla u}{|\nabla u|}) + \tfrac{1}{2} (\tfrac{3}{2} p^2 - 2
p\ensuremath{\operatorname{tr}}_g k + 2 \langle \nabla p, \tfrac{\nabla u}{|\nabla u|}
\rangle) \\
\geq & \mu - |J| + \tfrac{1}{2} (\tfrac{3}{2} p^2 - 2
p\ensuremath{\operatorname{tr}}_g k - 2| \nabla p|) \\
> & 0,
\end{split}
\end{align}
  by \eqref{most strict dec}. Each regular level set of $u$ in \eqref{perturbation of
  spacetime harmonic} is copies of topological torus, applying Gauss-Bonnet
  theorem on each regular level set, we have the right hand side of
  \eqref{integral inequality of spacetime harmonic function} is strictly
  positive.
  
  We study now the boundary terms in \eqref{integral inequality of spacetime
  harmonic function}. On $\partial_+M$, $D u$ is nonzero and $\tfrac{D u}{|D u|}$ points outward by the strong
  maximum principle, so $\nabla u = |\nabla u| \nu_+$. Using \eqref{spacetime harmonic
  function} and the decomposition of the Laplacian $\Delta_g$,
  \[ (-\ensuremath{\operatorname{tr}}_g k + \tfrac{3}{2} p) |\nabla u| = \Delta_g u
     = H_+  \langle \nabla u, \nu_+ \rangle + (\nabla^2 u) (\nu_+, \nu_+) . \]
  So
\begin{align}
\begin{split}
& \partial_{\nu_+} |\nabla u| + \tilde{k} (\nabla u, \nu_+) \\
= & \tfrac{1}{|\nabla u|} (\nabla^2 u) (\nabla u, \nu_+) + |\nabla u| k (\nu_+, \nu_+) -
\tfrac{1}{2} p |\nabla u| \\
= & (\nabla^2 u) (\nu_+, \nu_+) + |\nabla u| k (\nu_+, \nu_+) - \tfrac{1}{2} p |\nabla u|
\\
= & - H_+  \langle \nabla u, \nu_+ \rangle + (-\ensuremath{\operatorname{tr}}_g
k + \tfrac{3}{2} p) |\nabla u| + |\nabla u| k (\nu_+, \nu_+) - \tfrac{1}{2} p |\nabla u|
\\
= & - H_+ |\nabla u| -\ensuremath{\operatorname{tr}}_{\partial_+} k |\nabla u| + p
|\nabla u| \\
= & (- \theta_+ + p) |\nabla u| .
\end{split}
\end{align}
  Similarly on $\partial_-M$, $D u = |\nabla u| \nu_-$ and
  \begin{equation}
    \partial_{\nu_-} |\nabla u| + \tilde{k} (\nabla u, \nu_-) = - (\theta_- - p) |\nabla u| .
  \end{equation}
  So the left hand side of \eqref{integral inequality of spacetime harmonic
  function}
  \begin{equation}
    \int_{\partial_{\pm}M} \pm \partial_{\nu \pm} |\nabla u| + k (\nabla u, \pm
    \nu_{\pm}) = \int_{\partial_{\pm}M} \mp (\theta_{\pm} - p) |\nabla u| < 0
  \end{equation}
  is negative by \eqref{strict expansion bound}. We have a contradiction. So
  we have the width of $(M, g, k)$ is less than $t_+ - t_-$.
\end{proof}

\begin{remark}
  The band width estimate via the spacetime harmonic function of the case $k =
  0$ is due to the work (in preparation) of S. Hirsch, D. Kazaras, M. Khuri and
  Y. Zhang.
\end{remark}
Now we turn to study the rigidity.

\begin{lemma}
  \label{rigid width function}If $\ensuremath{\operatorname{width}} (M, g) =
  t_+ - t_-$, then there exists a function $\phi : M \to [t_-, t_+]$ such that
  $\phi (\partial_{\pm}M) = t_{\pm}$ with $\ensuremath{\operatorname{Lip}}
  (\phi) \leq 1$.
\end{lemma}

\begin{proof}
  Let
  \[ \phi (x) = \min \{t_- +\ensuremath{\operatorname{dist}}(x, \partial_-M),
     t_+ \} . \]
  It is easy to see that $\ensuremath{\operatorname{Lip}} (\phi) \leq 1$
  and $\phi (\partial_-M) = t_-$. Since $\ensuremath{\operatorname{width}} (M,
  g) = t_+ - t_-$, so the distance of any $x \in \partial_+M$ to the boundary
  $\partial_-M$ is greater or equal to $t_+ - t_-$. Therefore
  \[ t_- +\ensuremath{\operatorname{dist}} (x, \partial_-M) \geq t_- + t_+
     - t_- = t_+, \]
  implying that $\phi (\partial_+M) = t_+$.
\end{proof}

\begin{remark}
  We can also apply the Arzela-Ascoli lemma to a family of smooth functions
  constructed in {\cite[Lemma 4.1]{zhu-width-2021}} to obtain a $\phi$
  with the same properties. It is not as explicit.
\end{remark}

With the help of Lemma \ref{rigid width function}, we have a rigidity result.

\begin{theorem}
  \label{rigidity positive case}Assume that $(M^3, g, k)$ satisfies the
  conditions in Theorem \ref{width estimate positive} and moreover
  $\ensuremath{\operatorname{width}} (M, g) = t_+ - t_-$. Then $\phi$
  constructed in Lemma \ref{rigid width function} is
  \[ \phi (x) =\ensuremath{\operatorname{dist}} (x, \partial_-M) + t_- . \]
\end{theorem}

\begin{proof}
  Let $p = \eta \circ \phi$, then the \eqref{modified dec} and the bounds
  \eqref{expansion bound} on the null expansion of the boundaries are
  satisfied. Using the previous proof of Theorem \ref{width estimate positive}
  with a solution of \eqref{perturbation of spacetime harmonic} \ (which does
  not depend on these curvature bounds), we have the following integral
  inequality,
\begin{align}
\begin{split}
0 \geq & \int_{\partial_{\pm}} \mp (\theta_{\pm} - p) |\nabla u|
\\
\geq & \int_{-1}^{1} \int_{\Sigma_t} [\tfrac{1}{2} \tfrac{|\nabla^2 u -
k |\nabla u| + \tfrac{1}{2} p g |^2}{|\nabla u|^2} + \mu + J (\tfrac{\nabla
u}{|\nabla u|}) - K_{\Sigma_t}] \\
& + \int_{-1}^{1} \int_{\Sigma_t} \tfrac{1}{2} (\tfrac{3}{2} p^2 - 2
p\ensuremath{\operatorname{tr}}_g k + 2 \langle \nabla p, \tfrac{\nabla u}{|\nabla u|}
\rangle) \\
\geq & 0,
\end{split}
\end{align}
  by Gauss-Bonnet theorem on the level set $\Sigma_t$. Hence we get
  $\theta_{\pm} = p = \eta (t_{\pm})$ on $\partial_{\pm}M$, $\mu + J (\tfrac{\nabla
  u}{|\nabla u|}) = \tfrac{1}{2} \sigma$ and
  \begin{equation}
    \tfrac{3}{2} p^2 - 2 p\ensuremath{\operatorname{tr}}_g k + 2 \langle \nabla h,
    \tfrac{\nabla u}{|\nabla u|} \rangle = - \sigma .
  \end{equation}
  Since $\eta$ satisfies \eqref{ode}, we deduce that
  \begin{equation}
    \eta' \circ \phi \langle \nabla \phi, \tfrac{\nabla u}{|\nabla u|} \rangle = \langle \nabla p,
    \tfrac{\nabla u}{|\nabla u|} \rangle = \eta' \circ \phi .
  \end{equation}
  By $\eta' < 0$, we have that $\nabla \phi$ is parallel to $\nabla u$ and $|\nabla \phi | =
  1$ proving that $\phi$ is a distance function, and because $\phi
  (\partial_-M) = t_-$, we have $\phi = t_- +\ensuremath{\operatorname{dist}}
  (\cdot, \partial_-M)$.
\end{proof}

Actually, we can assert that each level set of $u$ coincide with a level set
of $\phi$ and each level set is a flat, stable torus of prescribed null
expansion $\eta \circ \phi$. For this, we refer the readers to the work of
Tsang {\cite{tsang-dihedral-2021-arxiv}} on the Gromov diheral rigidity of
initial data sets.

\section{Proof via the Dirac operator}\label{spinor proof}

In this section, following the method in \cite{cecchini-scalar-2021-arxiv}, we give a proof of Theorem \ref{width estimate positive} via the Dirac operator. 
\subsection{Callias operator} In this subsection, we review basics of Callias operators; one can refer to \cite{cecchini-scalar-2021-arxiv}. 

The following definition of the Dirac bundle can be found in \cite[Definition 5.2]{lawson-spin-1989}.
\begin{definition}[Dirac bundle]
A ($\mb{Z}_2$-graded) {\it Dirac bundle} over $M$ is a Hermitian vector bundle $S\to M$ with a metric connection $\n:C^{\infty}(M,S)\to C^{\infty}(M,T^*M\otimes S)$ (endowed with a parallel and orthogonal $\mb{Z}_2$-grading $S=S^+\oplus S^-$) and a parallel bundle map $c:T^*M\to \mr{End}(S)$, called {\it Clifford multiplication}, such that $c(\omega)$ is anti-self-adjoint (and odd), and $c(\omega)^2=-|\omega|^2$ for all $\omega\in T^*M$. 
\end{definition}
To define the Callias operator, we also need the following definitions of relative Dirac bundle and admissible potential, see \cite[Definitions 2.2 and 3.1]{cecchini-scalar-2021-arxiv}.

\begin{definition}[Relative Dirac bundle]
	Let $K\subset M^\circ$ be compact subset in the interior. A {\it relative Dirac bundle} with {\it support} K is a $\mb{Z}_2$-graded Dirac bundle $S\to M$ together with an odd, self-adjoint, parallel bundle involution $\sigma\in C^{\infty}(M\backslash K,\mr{End}(S))$ satisfying $c(\omega)\sigma=-\sigma c(\omega)$ for every $\omega\in T^*M|_{M\backslash K}$ and such that $\sigma$ admits a smooth extension to a bundle map on an open neighborhood of $\o{M\backslash K}$.
\end{definition}

\begin{definition}[Admissible potential]
	A Lipschitz function $\psi:M\to \mb{R}$ is called an {it admissible potential} if $\psi=0$ on $K$ and there exists a compact set $K\subset L\subset M$ such that $\psi$ is equal to a nonzero constant on each component of $M\backslash L$.
\end{definition}

Let $\psi:M\to \mb{R}$ be an admissible potential and  $S\to M$ be a relative Dirac bundle. Then we can define the Callias operator as follows, see \cite[(3.1)]{cecchini-scalar-2021-arxiv}.

\begin{definition}[Callias operator]
The {\it Callias operator} is defined as 
\[
  \mc{B}_\psi:=D+\psi\sigma,
\]
where $\psi$ is an admissible potential, and $D$ is the Dirac operator associated to the metric connection $\n$.
\end{definition}

\begin{example}[{\cite[Example 2.6]{cecchini-scalar-2021-arxiv}}]\label{ex1}
Let $(M,g)$ be an $n$-dimensional Riemannian spin band and let $\s{S}_M\to M$ be the associated complex spinor bundle endowed with the connection induced by the Levi-Civita connection.  Let $E\to M$ be a Hermitian bundle equipped with a metric connection. Then $S:=(\s{S}_M\otimes E)\oplus (\s{S}_M\otimes E)$ is a $\mb{Z}_2$-graded Dirac bundle with Clifford multiplication 
\begin{equation*}
  c:=\begin{pmatrix}
  0&c_{\s{S}}\otimes \operatorname{id}_E \\
  c_{\s{S}}\otimes \operatorname{id}_E&0 
\end{pmatrix}
\end{equation*}
 	Moreover, $S$ turns into a relative Dirac bundle with the involution
 	\begin{equation*}
  \sigma=\begin{pmatrix}
  0&-i \\
  i&0 
\end{pmatrix}
\end{equation*}
globally defined on $M$. The Dirac operator on $S$ is given by
\begin{equation*}
  \mc{D}=\begin{pmatrix}
  0&\s{D}_E \\
  \s{D}_E&0 
\end{pmatrix}
\end{equation*}
where $\s{D}_E:C^\infty(M,\s{S}_M\otimes E)\to C^\infty(M,\s{S}_M\otimes E)$ is the spinor Dirac operator on $(M,g)$ twist with the bundle $E$. Then the curvature term is 
\begin{equation}\label{cur}
  \mc{R}=\frac{\mr{scal}_g}{4}+\mc{R}^E,
\end{equation}
where $\mc{R}^E=\sum_{i<j}c(e^i)c(e^j)(\operatorname{id}_{\s{S}_M}\otimes R^{\n^E})$.
\end{example}
\subsection{Spectral estimates}
Let $(M,g,k)$ be an spin initial data set. Recall that
\begin{equation*}
  2\mu=R_g-|k|^2_g+(\mr{tr}_gk)^2,\quad  J_i=(\mr{div}_g k-d(\mr{tr}_gk))_i.
\end{equation*}
Let $E\to M$ be a Hermitian vector bundle with a metric connection, and let $\n$ be the induced metric connection on the Dirac bundle $S:=(\s{S}_M\otimes E)\oplus(\s{S}_M\otimes E)$. According to Example \ref{ex1}, $S$ turns into a relative  Dirac bundle with involution $\sigma$. The associated Dirac operator is given by
\begin{equation*}
  D=c(e^i)\n_{e_i}.
\end{equation*}
 Now we define a new connection on $S$ by 
 \begin{equation*}
  \wt{\n}_{e_i}=\n_{e_i}-\tfrac{1}{2}k_{ij}c(e^j)\sigma.
\end{equation*}
It is a connection on $S$ since $\tfrac{1}{2}k_{ij}c(e^j)\sigma\in A^1(M,\mr{End}(S))$. Set
\begin{equation*}
  \wt{D}=c(e^i)\wt{\n}_{e_i}.
\end{equation*}
By a direct calculation, one has
\begin{proposition}\label{prop1}
For any $u\in\Gamma(S)$, we have the following
\begin{itemize}
  \item[(i)] $\wt{\n}_{e_i}u=\n_{e_i}u+\frac{1}{2}k_{ji}\sigma c({e}^j)u$, $\wt{D}u=Du+ \tfrac{\mr{tr}_gk}{2}\sigma(u)$,
  \item[(ii)] $\n^*_{e_i}u=-\n_{e_i}u$, $\wt{\n}^*_{e_i}u=-\wt{\n}_{e_i}u-k_{ij}c({e}^j)\sigma(u)$, $D^*u=Du$, $\wt{D}^*u=\wt{D}u$;
  \item[(iii)] $D^2u=\n^*\n u+\mc{R}u$,$\wt{D}^2u=\wt{\n}^*\wt{\n}u+\wt{\mc{R}}u$, $\wt{\mc{R}}=\frac{1}{2}(\mu_E-J_i c(e^i)\sigma)$.
\end{itemize}
where $\mc{R}=\frac{R_g}{4}+\mc{R}^E$ and $\mu_E=\mu+2\mc{R}^E$.
\end{proposition}
\begin{remark}
	The proof of Proposition \ref{prop1} is the same as the identities for the Dirac-Witten operator because the endomorphism $c(e^0)$ is precisely an involution in a relative Dirac bundle. One can refer to \cite[Section III]{hijazi-dirac-witten-2003} for the related calculations of the Dirac-Witten operators.
\end{remark}

The Green's formula for the Dirac operator $\wt{D}$ is given by
$$
\int_{M}\langle\widetilde{D} u, v\rangle=\int_{M}\langle u, \widetilde{D} v\rangle+\int_{\partial M}\left\langle u, c\left(\nu^{b}\right) v\right\rangle,
$$
where $u, v \in C_{c}^{\infty}(M, S)$. Here $\nu$  denotes the normal vector field pointing inward. For the connection $\wt{\n}$, we have the following Green's formula
$$
\int_{M}\left\langle u, \widetilde{\nabla}^{*} \widetilde{\nabla} u\right\rangle=\int_{M}\langle\widetilde{\nabla} u, \widetilde{\nabla} u\rangle+\int_{\partial M}\left\langle u, \widetilde{\nabla}_{\nu} u\right\rangle .
$$
Hence
\begin{align*}
\begin{split}
  \int_M |\wt{D}u|^2&=\int_M \left\langle u,\wt{D}^2u\right\rangle+\int_{\p M} \left\langle u,c(\nu^\flat)\wt{D}u\right\rangle\\
  &=\int_M \left\langle u,\wt{\n}^*\wt{\n}u\right\rangle+\int_M \left\langle u,\wt{\mc{R}}u\right\rangle+\int_{\p M} \left\langle u,c(\nu^\flat)\wt{D}u\right\rangle\\
  &=\int_M |\wt{\n}u|^2+\int_M \left\langle u,\wt{\mc{R}}u\right\rangle+\int_{\p M}\left\langle u,c(\nu^\flat)\wt{D}u+\wt{\n}_{\nu}u\right\rangle.
 \end{split}
\end{align*}
The boundary Dirac operator is defined as follows
\begin{align*}
\begin{split}
  \mc{A}:=\sum_{i=1}^{n-1}c^\p(e^i)\n^\p_{e_i},
 \end{split}
\end{align*}
where $e_n=-\nu$ and $c^\p(e^i)=c(e^i)c(\nu^\flat)$ and $\n^{\p}_{e_i}=\n_{e_i}+\frac{1}{2}c^\p(\n_{e_i}\nu^{\flat})$, see e.g. \cite[Section 2]{cecchini-scalar-2021-arxiv}. Denote by $p$ the second fundamental form on $\p M$, then $\mr{tr}_{\p M}h=\sum_{i=1}^{n-1}\left\langle e_i,\n_{e_i}(-\nu)\right\rangle$.
 Then 
\begin{align*}
\begin{split}
  \mc{A}&=\frac{1}{2}\mr{tr}_{\p M}h-c(\nu^\flat) D-\n_{\nu}\\
  &=\frac{1}{2}\mr{tr}_{\p M}h-c(\nu^b)(\wt{D}-\frac{\mr{tr}_gk}{2}\sigma)-(\wt{\n}_{\nu}+\frac{1}{2}k_{j\nu}c({e}^j)\sigma)\\
  &=\frac{1}{2}\mr{tr}_{\p M}h+\frac{\mr{tr}_gk}{2}c(\nu^\flat)\sigma-\frac{1}{2}k_{\nu\nu}c(\nu^\flat)\sigma-\frac{1}{2}k_{a\nu}c( e^a)\sigma-(c(\nu^\flat)\wt{D}+\wt{\n}_{\nu})\\
  &=\frac{1}{2}\mr{tr}_{\p M}h+\frac{\mr{tr}_{\p M}k}{2}c(\nu^\flat)\sigma-\frac{1}{2}k_{a\nu}c( e^a)\sigma-(c(\nu^\flat)\wt{D}+\wt{\n}_{\nu})
 \end{split}
\end{align*}
which follows that
\begin{align*}
\begin{split}
   \int_M |\wt{D}u|^2&=\int_M |\wt{\n}u|^2+\int_M \left\langle u,\wt{\mc{R}}u\right\rangle\\
   &\quad +\int_{\p M}\left\langle u,(\frac{1}{2}\mr{tr}_{\p M}h+\frac{\mr{tr}_{\p M}k}{2}c(\nu^\flat)\sigma-\frac{1}{2}k_{a\nu}c( e^a)\sigma-\mc{A})u\right\rangle.
 \end{split}
\end{align*}
The Penrose operator of $\wt{\n}$ is defined as 
\begin{align*}
\begin{split}
  \wt{\mc{P}}_\xi u=\wt{\n}_\xi u+\frac{1}{n}c(\xi^\flat) \wt{D}u.
 \end{split}
\end{align*}
The Friedrich inequality is then 
\begin{align*}
\begin{split}
|\wt{\n}u|^2-\frac{1}{n}|\wt{D}u|^2 = |\wt{\mc{P}}u|^2\geq 0,
 \end{split}
\end{align*}
which follows that
\begin{align*}
\begin{split}
 &  \int_M |\wt{D}u|^2=\frac{n}{n-1}\int_M |\wt{\mc{P}}u|^2+\frac{n}{n-1}\int_M \left\langle u,\wt{\mc{R}}u\right\rangle\\
   &+\frac{n}{n-1}\int_{\p M}\left\langle u,(\frac{1}{2}\mr{tr}_{\p M}h+\frac{\mr{tr}_{\p M}k}{2}c(\nu^\flat)\sigma-\frac{1}{2}k_{a\nu}c( e^a)\sigma-\mc{A})u\right\rangle.
 \end{split}
\end{align*}
Let $\psi:M\to\mb{R}$ be a Lipschitz function such that $2\psi+\mr{tr}_gk$ is an admissible potential, then 
\begin{equation*}
{  \mc{B}_{\psi}=\wt{D}+\psi \sigma=D+(\frac{1}{2}\mr{tr}_gk+\psi)\sigma}
\end{equation*}
is a Callias operator.
For any $u\in C^\infty_c(M,\mb{S})$, one has
\begin{multline}
  \int_M | \mc{B}_{\psi}u|^2=\int_M (|\wt{D}u|^2+|\psi|^2|u|^2)\\
  +\int_M\left\langle u,(c(\mathrm{d}\psi)\sigma+\psi (\mr{tr}_gk) )u\right\rangle+\int_{\p M}\left\langle u,\psi c(\nu^b)\sigma(u)\right\rangle.
\end{multline}
Thus 
\begin{align*}
\begin{split}
&   \int_M | \mc{B}_{\psi}u|^2=\frac{n}{n-1}\int_M |\wt{\mc{P}}u|^2+\frac{n}{n-1}\int_M \left\langle u,\wt{\mc{R}}u\right\rangle\\
   & +\frac{n}{n-1}\int_{\p M}\left\langle u,(\frac{1}{2}\mr{tr}_{\p M}h+\frac{\mr{tr}_{\p M}k}{2}c(\nu^\flat)\sigma-\frac{1}{2}k_{a\nu}c( e^a)\sigma-\mc{A})u\right\rangle\\
   &+\int_M |\psi|^2|u|^2+\left\langle u,(c(\mathrm{d}\psi)\sigma+\psi \mr{tr}_gk )u\right\rangle+\int_{\p M}\left\langle u,\psi c(\nu^b)\sigma(u)\right\rangle.
 \end{split}
\end{align*}
Using the fact $|\wt{\mc{P}}u|\geq 0$ and $\left\langle u, c(\mathrm{d}\psi)\sigma u\right\rangle\leq |\mathrm{d}\psi||u|^2$, so 
\begin{align*}
\begin{split}
   &\int_M | \mc{B}_{\psi}u|^2\geq \int_M (|\psi|^2-|\mathrm{d}\psi|+\psi \mr{tr}_gk)|u|^2+\frac{n}{n-1}\int_M \left\langle u,\wt{\mc{R}}u\right\rangle\\
   &+\frac{n}{n-1}\int_{\p M}\left\langle u,(\frac{1}{2}\mr{tr}_{\p M}h+(\frac{\mr{tr}_{\p M}k}{2}+\frac{n-1}{n}\psi)c(\nu^\flat)\sigma-\frac{1}{2}k_{a\nu}c( e^a)\sigma-\mc{A})u\right\rangle.
 \end{split}
\end{align*}
Let $s:\p M\to \{\pm 1\}$ be a choice of signs, and we consider the following boundary condition
\begin{equation*}
  \{u\in C^\infty_c(M,S)| \chi(u|_{\p M})=u|_{\p M}\},
\end{equation*}
where $\chi:=sc(\nu^b)\sigma$ is the boundary chirality. Then $\chi^2=\mr{Id}$ and 
\begin{equation*}
  \chi\mc{A}=-\chi\mc{A}, \quad \chi(c(e^a)\sigma)=-(c(e^a)\sigma)\chi.
\end{equation*}
Hence
\begin{equation*}
  \left\langle u,\mc{A}u\right\rangle=0=\left\langle u,c(e^a)\sigma u\right\rangle.
\end{equation*}
Under this boundary condition, one has
\begin{align*}
  \int_M | \mc{B}_{\psi}u|^2 &\geq  \int_M (|\psi|^2-|\mathrm{d}\psi|+\psi \mr{tr}_gk)|u|^2+\frac{n}{n-1}\int_M \left\langle u,\wt{\mc{R}}u\right\rangle\\
  &\quad +\frac{n}{n-1}\int_{\p M}\left\langle u,(\frac{1}{2}\mr{tr}_{\p M}h+(\frac{\mr{tr}_{\p M}k}{2}+\frac{n-1}{n}\psi)c(\nu^\flat)\sigma)u\right\rangle.
\end{align*}
Since $$\left\langle u,\wt{\mc{R}}u\right\rangle\geq \frac{1}{2}(\mu-|J|)|u|^2-\|\mc{R}^E\|_{\infty}|u|^2,$$ where $\|\mc{R}^E\|_{\infty}:=\sup_{\|u\|=1} \left\{\left\langle\mathcal{R}^{E}u, u\right\rangle\right\}$, so 
\begin{align*}
&\quad \int_M | \mc{B}_{\psi}u|^2\\
&\geq  \frac{n}{n-1}  \int_M \left(\frac{n-1}{n}(|\psi|^2-|\mathrm{d}\psi|+\psi \mr{tr}_gk)+\frac{1}{2}(\mu-|J|)-\|\mc{R}^E\|_{\infty}\right)|u|^2\\
  &\quad +\frac{n}{n-1}\int_{\p M}\left\langle u,(\frac{1}{2}\mr{tr}_{\p M}h+(\frac{\mr{tr}_{\p M}k}{2}+\frac{n-1}{n}\psi)s)u\right\rangle.
\end{align*}
Now we set $\wt{\psi}=-\frac{2(n-1)}{n}\psi$ and $\theta_{\pm}=\pm \mr{tr}_{\p M} h+\mr{tr}_{\p M}k$ and $s=\pm 1$ on $\p_\pm M$, thus
\begin{align}\label{estB}
\begin{split}
&\quad \int_M | \mc{B}_{\psi}u|^2 \\
&\geq  \frac{n}{2(n-1)}  \int_M \left[\mu-|J|+\tfrac{1}{2}\left(\tfrac{n}{n-1}\wt{\psi}^2 -2\wt{\psi}\mr{tr}_gk-2|\mathrm{d}\wt{\psi}| \right)-2\|\mc{R}^E\|_{\infty}\right]|u|^2\\
  &\quad +\frac{n}{2(n-1)}\int_{\p_+ M}(\theta_+-\wt{\psi})|u|^2+\frac{n}{2(n-1)}\int_{\p_-M}(-\theta_-+\wt{\psi})|u|^2.
  \end{split}
\end{align}
\begin{remark}
For $k=0$, \eqref{estB} is reduced to \cite[Theorem 4.3 and (4.2)]{cecchini-scalar-2021-arxiv}.	
\end{remark}
Now we denote 
\begin{equation*}
  \mr{H}^1_{s}(M,S):=\{u\in \mr{H}^1_{\mr{loc}}(M,S)\cap \mr{L}^2(M,S)|\wt{D}u\in \mr{L}^2(M,S), \chi(u|_{\p M})=u|_{\p M}\}. 
\end{equation*}
Then we get
\begin{theorem}\label{thm1}
Let $(M,g,k)$ be a CMC initial data set satisfying the conditions in Theorem \ref{width estimate positive}.
For any $\epsilon>0$, if there exists a Hermitian bundle $E$ such that $\|\mc{R}^E\|_\infty<\epsilon$  and $\mr{H}^1_{s}(M,S)\cap \ker\mc{B}_\psi\neq \{0\}$, then 
\begin{equation*}
  \mr{width}(M,g)\leq t_+-t_-.
\end{equation*}
\end{theorem}
\begin{proof}
By Lemma \ref{construction of h}, if $  \mr{width}(M,g)> t_+-t_-$, we can find $\wt{\psi}$ such that 
	\begin{equation*}
  \mu-|J|+\tfrac{1}{2}\left(\frac{n}{n-1}\wt{\psi}^2-2\wt{\psi}\mr{tr}_gk -2|\mathrm{d}\wt{\psi}|\right)>\epsilon>2\|\mc{R}^E\|_\infty
\end{equation*}
and 
\begin{equation*}
  \theta_--\wt{\psi}<0 \quad \text{on} \quad \p_-M 
\end{equation*}
\begin{equation*}
  \theta_+-\wt{\psi}>0 \quad \text{on} \quad \p_+M 
\end{equation*}
which follows that
\begin{equation*}
 0=\int_M | \mc{B}_{\psi}u|^2>0,
\end{equation*}
contradiction. The proof is complete.
\end{proof}

\begin{remark}
We can choose a suitable $\wt{\psi}$ such that Lemma \ref{construction of h} holds and $-\frac{n}{n-1}\wt{\psi}+\lambda$ is an admissible potential. Note that
\begin{align*}
\begin{split}
 2 \psi+\mr{tr}_g(k)&=-\frac{n}{n-1}\wt{\psi}+\lambda=-\frac{n}{n-1}\eta\circ\phi+\lambda,
  \end{split}
\end{align*}
where $\phi:(M,g)\to [t_--\epsilon,t_++\epsilon]$ and $\mr{width}(M,g)>t_+-t_-+2\epsilon$, $\phi^{-1}(t_--\epsilon)=\p_-M$ and $\phi^{-1}(t_-+\epsilon)=\p_+M$, see Lemma \ref{construction of h} and $\eta'<0$. The Lipschitz function $\phi$ can be chosen such that $\phi$ is constant near $\p_\pm M$.
\begin{figure}[ht]
\centering
\includegraphics[width=0.5\textwidth]{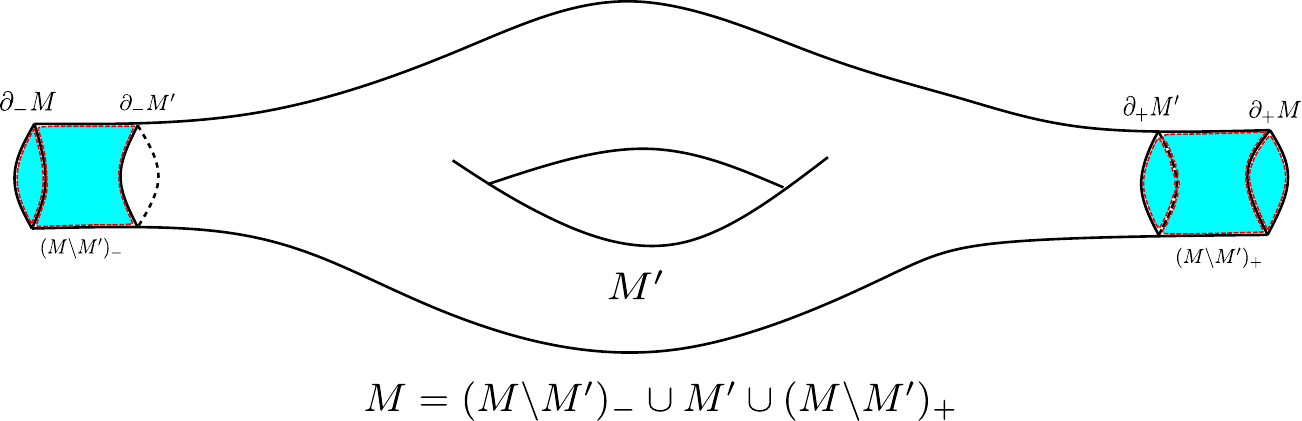}
\caption{The construction of $(M',g')$}
\end{figure}
 In fact, let $(M',g')$ be a Riemannian manifold by cutting a small collar neighborhood of $\p_\pm M$ from $(M,g)$, such that
\begin{equation*}
  \mr{width}(M',g')>t_+-t_-+2\epsilon.
\end{equation*}
Then there exists a smooth function $\phi':(M',g)\to [t_--\epsilon,t_++\epsilon]$ with $\phi'^{-1}(t_--\epsilon)=\p_-M'$ and $\phi'^{-1}(t_-+\epsilon)=\p_+M'$. Now we set
\begin{equation*}
  \phi(p)= \begin{cases}
 t_--\epsilon	& p\in(M\backslash M')_-\\
 	\phi'&p\in M'\\
 	t_++\epsilon& p\in (M\backslash M')_+.
 \end{cases}
\end{equation*}
Then $\phi:(M,g)\to [t_--\epsilon,t_++\epsilon]$ is a Lipschitz function such that  $\p_-M\subset \phi^{-1}(t_--\epsilon)$ and $\p_+M\subset\phi^{-1}(t_-+\epsilon)$, which is constant near $\p_{\pm}M$. By taking $(M\backslash M')_{\pm}$ small enough, so $\phi$ satisfies Lemma \ref{construction of h}.

Therefore, $2 \psi+\mr{tr}_g(k)$ is a nonzero constant for some small $\epsilon>0$.   Hence  $\mc{B}_\psi=\wt{D}+\psi \sigma$ is  a Callias operator.  
\end{remark}

\begin{remark}
Inspired by the rigidity theorem \cite[Theorem 8.3]{cecchini-scalar-2021-arxiv}, we assume that  $(M,g,k)$ satisfies the following conditions:
\begin{itemize}
  \item[(i)]  $\mu-|J|-2\|\mc{R}^E\|_\infty\geq \frac{1}{2}\sigma$;
  \item[(ii)] $x:M\to [t_-,t_+]$ is a smooth width function for some $t_-,t_+\in \mb{R}$;
  \item[(iii)] $\theta_+\geq\eta(t_+)$ and $\theta_-\leq \eta(t_-)$;
  \item[(iv)] $\ker\mc{B}_{\psi,s}\neq\{0\}$. 
\end{itemize}
For any  $u\in \ker\mc{B}_{\psi,s}$ with $u\neq 0$, we have
\begin{align*}
  0&=\int_M|\mc{B}_{\psi}u|^2\\
  &\geq \frac{n}{n-1}\int_M |\wt{\mc{P}}u|^2\\
  &\quad+\frac{n}{2(n-1)}\int_M (-\wt{\psi}')\left\langle u,(c(\mathrm{d}x)\sigma+1)u\right\rangle\\
  &\quad+\frac{n}{2(n-1)}  \int_M \left(\mu-|J|-2\|\mc{R}^E\|_{\infty}-\frac{1}{2}\sigma\right)|u|^2\\
   &\quad+\frac{n}{2(n-1)}  \int_M \left(\frac{1}{2}\sigma+\frac{n}{2(n-1)}\wt{\psi}^2+\wt{\psi}'-\wt{\psi}\mr{tr}_gk\right)|u|^2\\
  &\quad +\frac{n}{2(n-1)}\int_{\p_+ M}(\theta_+-\wt{\psi})|u|^2+\frac{n}{2(n-1)}\int_{\p_-M}(-\theta_-+\wt{\psi})|u|^2\\
  &\geq 0
\end{align*}
by the above assumptions. We obtain
\begin{equation*}
  \wt{\mc{P}}_\xi u=\wt{\n}_\xi u+\frac{1}{n}c(\xi^\flat) \wt{D}u=0
\end{equation*}
on $M$ and 
\begin{equation*}
  c(\mathrm{d}x)\sigma u=-u
\end{equation*}
almost everywhere. Since $\mc{B}_\psi u=0$, so  
\begin{equation*}
    (\wt{\n}_\xi -\frac{\psi}{n}c(\xi^\flat)\sigma)u=0.
\end{equation*}
Note that the boundary elliptic regularity implies that $u$ is smooth. 
If $u$ vanishes at some point $p_0\in M$, for any point $p_1\in M$, let $\gamma=\gamma(t)$ be a smooth path in $M$ connected $p_0$ and $p_1$, then 
\begin{equation*}
    (\wt{\n}_{\gamma'(t)} -\frac{\psi}{n}c(\gamma'(t)^\flat)\sigma)u=0.
\end{equation*}
which is a linear ordinary differential equation, and has the unique solution $u\equiv 0$ along $\gamma(t)$. Thus $u\equiv 0$ on $M$ since $M$ is connected, which contradicts to $u\neq 0$. Hence $|u|>0$. The argument to deduce $|u|>0$ also can be found in the proof of \cite[Lemma 2.6]{cecchini-Zeidler-2021}. Thus
 $$\mu-|J|-2\|\mc{R}^E\|_\infty= \frac{1}{2}\sigma,\quad \theta_\pm=\eta(t_\pm).$$
 
\end{remark}

\subsection{Band width estimates}
 
 In this subsection, we give some examples of bands satisfying the assumptions of Theorem \ref{thm1}. 
 
 \subsubsection{Infinite vertical $\widehat{\mr{A}}$-area}
 
 The following definition of {\it infinite vertical $\widehat{\mr{A}}$-area} can be found in \cite[Definition 7.3]{cecchini-scalar-2021-arxiv}. 
 
 \begin{definition}[Infinite vertical $\widehat{\mr{A}}$-area]
 A band $M$ is said to have {\it infinite vertical $\widehat{\mr{A}}$-area}, if for every $\varepsilon>0$, there exists a Hermitian vector bundle $E\to M$ such that $\|R^E\|_{\infty}<\varepsilon$ and such that we have
 \begin{equation*}
  \int_{\p_-M}\widehat{\mathbf{A}}(\p_-M)\wedge \mr{ch}(E|_{\p_-M})\neq 0.
\end{equation*}
 \end{definition}
\begin{remark}
If $X$ is a closed spin manifold (even dimensional) with infinite $\widehat{A}$-area, then $M=X\times [-1,1]$ has infinite vertical $\widehat{\mr{A}}$-area. An important class of examples consists of even-dimensional compactly enlargeable manifolds, for example, $X=\mb{T}^{2m}$. On the other hand,  if $M$ is a $\widehat{\mr{A}}$-overtorical band, then $M$ also has infinite vertical $\widehat{\mr{A}}$-area.
\end{remark}
From Theorem \ref{thm1}, we obtain
\begin{theorem}\label{thminfA}
Let $(M,g,k)$ be a CMC initial data set satisfying the conditions in Theorem  \ref{width estimate positive}.
If $(M,g)$ is a spin band of infinite vertical $\widehat{\mr{A}}$-area, then 	
\begin{equation*}
  \mr{width}(M,g)\leq t_+-t_-.
\end{equation*}
\begin{proof}
From \cite[Corollary 3.10]{cecchini-scalar-2021-arxiv}, one has
\begin{equation*}
  \mr{ind}(\mc{B}_{\psi,s})=\mr{ind}(\s{D}_{\p_-M,E|_{\p_-M}})= \int_{\p_-M}\widehat{\mathbf{A}}(\p_-M)\wedge \mr{ch}(E|_{\p_-M})\neq 0,
\end{equation*}
	which follows that
	\begin{equation*}
   \ker\mc{B}_{\psi,s}=\mr{H}^1_{s}(M,S)\cap \ker\mc{B}_\psi\neq \{0\}.
\end{equation*}
\end{proof}
\end{theorem}
\begin{remark}
	For $k=0$ and $\sigma=n(n-1)$, 
	 Theorem \ref{thminfA} is exactly \cite[Theorem 7.6]{cecchini-scalar-2021-arxiv}.
\end{remark}

 \subsubsection{$\mc{K}\mc{O}$-bands}
 The following definition of $\mc{K}\mc{O}$ band can be found in {\cite[Page 9]{zeidler-width-2020}}.
 \begin{definition}[$\mc{K}\mc{O}$-band]
 	A band $M\in \mc{K}\mc{O}$ is called a $\mc{K}\mc{O}$ band if $M$ is spin and admits a flat bundle $E\to M$ of finitely generated projective Hilbert-A-modules for some unital Real $C^*$-algebra $A$ such that the twisted Dirac operator on $\p_{\pm}M$ has non-vanishing index $\mr{ind}(\s{D}_{\p_-M,E|_{\p_-M}})\neq 0\in \mr{KO}_{n-1}(A)$.
 \end{definition}
 \begin{remark}
In particular, all overtorical bands and $\widehat{\mr{A}}$-overtorical bands are $\mc{K}\mc{O}$ bands, see \cite[Proposition 5.2]{zeidler-width-2020}. 
If $M=X\times [-1,1]$ where $X$ is a closed spin manifold with nonvanishing Rosenberg index $\alpha(X)$, then $M$ is a $\mc{K}\mc{O}$-band, see \cite[Section 3]{zeidler-band-2020}.
\end{remark}

 Let $(M, g)$ be an $n$-dimensional Riemannian spin manifold. Let $A$ be a Real unital $C^{*}$-algebra and let $\left(E, \nabla^{E}\right)$ be a bundle of finitely generated projective Hilbert $A$-modules endowed with a metric connection. Then one can also define the {\it generalized Callias-type operator} $\mc{B}_{\psi,s}$, see \cite[Section 6]{cecchini} or \cite{cecchini-Zeidler}. Similarly, we can obtain the estimate \eqref{estB}, and Theorem \ref{thm1} holds for this flat bundle $E$. 
\begin{theorem}\label{thmKO}
Let $(M,g,k)$ be a CMC initial data set satisfying the conditions in Theorem \ref{width estimate positive}.
If $M$ is a $\mc{K}\mc{O}$-band, then 
\begin{equation*}
  \mr{width}(M,g)\leq t_+-t_-.
\end{equation*}
\end{theorem}
\begin{proof}
	If $M$ is a $\mc{K}\mc{O}$-band, then there exists a flat bundle $E\to M$ such that $\mr{ind}(\s{D}_{\p_-M,E|_{\p_-M}})\neq 0$. Hence 
	\begin{equation*}
    \mr{ind}(\mc{B}_{\psi,s})=\mr{ind}(\s{D}_{\p_-M,E|_{\p_-M}})\neq 0.
\end{equation*}
By Theorem \ref{thm1}, the proof is complete.
\end{proof}
In particular, we obtain a proof of Theorem \ref{width estimate positive} via the Dirac operator for bands $\mathbb{T}^{n-1}\times [-1,1]$ of all dimensions.

\section{Appendix}

In this section, we will discuss the solutions $\eta=\eta(t)$ of the ordinary differential equation
\begin{equation}
  \sigma + \tfrac{n}{n - 1} \eta^2 - 2 \eta \lambda + 2 \eta' = 0 ,\quad \eta'(t)<0.\label{ode}
\end{equation}
Let $p (x) = \tfrac{n}{n - 1} x^2 - 2 \lambda x + \sigma$, the solution to
\eqref{ode} depends on the sign of $\sigma - \tfrac{n - 1}{n} \lambda^2$. In
what follows, $c$ is a constant. We divide into three cases:

{{\bfseries{Case 1:}}} when $\sigma = \tfrac{n - 1}{n} \lambda^2$, $p$
has only one root. The ordinary differential equation \eqref{ode} reduces to
\[ \tfrac{n}{n - 1} (\eta - \tfrac{n - 1}{n} \lambda)^2 + 2 \eta' = 0. \]
The solution of \eqref{ode} is
\begin{equation}
    \eta (t) = \tfrac{n - 1}{n} \lambda + \tfrac{2 (n - 1)}{n (t - c)}\label{eta equal} 
\end{equation} 
on $\{t > c\}$ or $\{t < c\}$.

{{\bfseries{Case 2:}}} when $\sigma > \tfrac{n - 1}{n} \lambda^2$, it is
easy to see that
\[ p (x) = \tfrac{n}{n - 1} (x - \tfrac{n - 1}{n} \lambda)^2 + (\sigma -
   \tfrac{n - 1}{n} \lambda^2) \]
has no real root. We solve that
\begin{equation} \eta (t) = \tfrac{n - 1}{n} \lambda - \sqrt{\tfrac{n - 1}{n} \left( \sigma
   - \tfrac{n - 1}{n} \lambda^2 \right)} \tan \left( \tfrac{n}{2 (n - 1)}
   \sqrt{\tfrac{n - 1}{n} (\sigma - \tfrac{n - 1}{n} \lambda^2)} (t - c)
   \right) \label{eta greater}
\end{equation}   
with $t$ satisfying
\[ \left| \tfrac{n}{2 (n - 1)} \sqrt{\tfrac{n - 1}{n} (\sigma - \tfrac{n -
   1}{n} \lambda^2)} (t - c) \right| < \tfrac{\pi}{2} . \]

 {{\bfseries{Case 3:}}} when $\sigma < \tfrac{n - 1}{n} \lambda^2$, $p$
has two real roots $x_{\pm}$. Denote
$$a = \sqrt{\tfrac{n - 1}{n} (\tfrac{n -
1}{n} \lambda^2 - \sigma)}.$$ Then
\[ x_{\pm} = \tfrac{n - 1}{n} \lambda \pm a. \]
So \eqref{ode} turns into
\[ \tfrac{n}{n - 1} (\eta - \tfrac{n - 1}{n} \lambda + a) (\eta - \tfrac{n -
   1}{n} \lambda - a) = - 2 \eta' = - 2 (\eta - \tfrac{n - 1}{n} \lambda)' . \]
We get the solution
\[ \eta (t) = \tfrac{n - 1}{n} \lambda + a \coth (\tfrac{n}{2 (n - 1)} a (t -
   c)) \]
which is
\begin{equation} \eta (t) = \tfrac{n - 1}{n} \lambda + \sqrt{\tfrac{n - 1}{n} (\tfrac{n -
   1}{n} \lambda^2 - \sigma)} \coth \left( \tfrac{n}{2 (n - 1)} 
   \sqrt{\tfrac{n - 1}{n} (\tfrac{n - 1}{n} \lambda^2 - \sigma)} (t - c)
   \right) \label{eta less}
\end{equation}   
with $t > c$ or $t < c$.

\bibliographystyle{alpha}
\bibliography{band-width-estimates}
\end{document}